\documentclass[11pt, a4paper]{article}
\usepackage{amssymb,amsmath,amsfonts,amsthm}
\usepackage{enumitem}
\usepackage{bm}
\usepackage{fullpage}
\usepackage{xcolor}

\numberwithin{equation}{section}

\newtheorem{thm}{Theorem}[section]
\newtheorem{lemma}[thm]{Lemma}
\newtheorem{prop}[thm]{Proposition}

\theoremstyle{definition}
\newtheorem{defn}[thm]{Definition}

\theoremstyle{remark}
\newtheorem{rmk}[thm]{Remark}

\DeclareMathOperator{\Cov}{Cov}

\newcommand\bR{{\mathbb {R}}}
\newcommand\N{{\mathbb {N}}}
\newcommand\bN{{\mathbb {N}}}

\newcommand\bZ{{\mathbb {Z}}}

\newcommand\eps{\varepsilon}
\DeclareMathOperator{\E}{\mathbb{E}}
\DeclareMathOperator{\bE}{\mathbb{E}}
\DeclareMathOperator{\p}{ { \mathbb P} }
\DeclareMathOperator{\bP}{ { \mathbb P} }

\def\tg{{\tilde{g}}}

\def\cA{{\mathcal{A}}}

\title{Rates in almost sure invariance principle
for quickly mixing dynamical systems}

\author{
C.~Cuny{\footnote{Universit\'e de Brest, LMBA, UMR CNRS 6205. Email: christophe.cuny@univ-brest.fr}},
J.~Dedecker{\footnote{Universit\'e Paris Descartes, Sorbonne Paris Cit\'e,  Laboratoire MAP5 (UMR 8145). Email: jerome.dedecker@parisdescartes.fr}},
A.~Korepanov{\footnote{University of Exeter, UK. Email: a.korepanov@exeter.ac.uk}},
Florence Merlev\`ede{\footnote{Universit\'{e} Paris-Est, LAMA (UMR 8050), UPEM, CNRS, UPEC. Email: florence.merlevede@u-pem.fr}}
}
\date{21 November 2018}

\begin{document}

\maketitle

\begin{abstract}
  For a large class of quickly mixing dynamical systems, we prove
  that the error in the almost sure approximation with a Brownian motion
  is of order $O((\log n)^a)$ with $a \geq 2$. 
  Specifically, we consider nonuniformly expanding maps
  with exponential and stretched exponential decay of correlations,
  with one-dimensional H\"older continuous observables. 
\end{abstract}


\smallskip


\section{Introduction and main result}

Let $X$ be a bounded metric space and $f \colon X \to X$ be a transformation,
preserving a Borel probability measure $\mu$.
Suppose that $\varphi \colon X \to \bR$ is H\"older continuous with
$\int \varphi \, d\mu = 0$. We consider the Birkhoff sums
\begin{equation}
  \label{eq:Sn}
  S_n(\varphi)  = \sum_{k=0}^{n-1} \varphi \circ f^k
\end{equation}
as a discrete time random process, defined on the probability space $(X, \mu)$.
It is common in chaotic dynamical systems that $S_n(\varphi)$ behaves
like a Brownian motion.
For example, $S_n(\varphi)$ may satisfy Donsker's invariance principle:
the normalized process
$X_t = n^{-1/2} S_{\lfloor n t \rfloor} (\varphi)$
may converge weakly to a Brownian motion as $n \to \infty$.
A basic and natural question is, how \emph{close} is $S_n(\varphi)$ to a Brownian motion?
For many chaotic dynamical systems, $S_n(\varphi)$ can be almost surely
approximated by a Brownian motion. In this work we look at the error of such
approximations.

\begin{defn}
We say that a random process $(X_n)_{n \geq 0}$ satisfies the
\emph{almost sure invariance principle} (ASIP) with rate $o(r_n)$,
where $r_n$ is a deterministic sequence such that $r_n = o(n^{1/2} \log \log n)$,
if, possibly enlarging the probability space, there exists
a Brownian motion $W_t$ such that
\[
  X_n = W_n + o(r_n) \quad \text{almost surely.}
\]
\end{defn}

The ASIP was introduced by Strassen~\cite{S64} as a tool
to prove the functional law of iterated logarithm.
Besides that, the ASIP implies  Donsker's invariance principle and
a range of other laws, see for example Philipp and Stout~\cite{PS75}.

A question of particular interest and challenge is to identify the optimal
rate in the ASIP. Strassen conjectured that under the best realistic assumptions
(such as independent increments assuming values $\{-1,+1\}$ with probability $1/2$ each),
the best possible rate is $O(n^{1/4} (\log n)^{1/2} (\log \log n)^{1/4})$. Kiefer~\cite{K69} showed that, for martingales, this is the best rate one can  
obtain by mean of the Skorokhod embedding method. On the other hand,
better rates were proved possible by Cs\"org\H{o} and R\'ev\'esz~\cite{CR75}
and then by Koml\'os, Major and Tusn\'ady~\cite{KMT75} and Major \cite{Ma76}:

\begin{thm}[KMT approximation]
  \label{thm:KMT}
  Suppose that $X_n = \sum_{k=1}^n \xi_k$, where $(\xi_k)$ is a sequence of
  real valued independent and identically distributed random variables
  with $\bE \xi_1 = 0$. Then
  \begin{itemize}
    \item[(a)] if $\bE(|X_1|^p) < \infty$ with $p > 2$, 
      then $(X_n)$ satisfies the ASIP with rate $o(n^{1/p})$;
    \item[(b)] if $\bE(e^{\delta |X_1|}) < \infty$ with $\delta > 0$,
      then $(X_n)$ satisfies the ASIP with rate $O(\log n)$.
  \end{itemize}
\end{thm}

For processes with dependent increments, it took several decades to obtain better  rates than
$O(n^{1/4} (\log n)^{1/2} (\log \log n)^{1/4})$.
A number of different ways to establish the ASIP were found,
but those with good rates required independence of increments and
proved very hard to extend.

Recently the rate $o(n^\eps)$ for arbitrarily small $\eps > 0$ was reached
by Berkes, Liu and Wu~\cite{BLW14} for processes driven by a Bernoulli shift:
\[
  X_n = \sum_{k=1}^n \psi (\ldots, \xi_{k-1}, \xi_k, \xi_{k+1}, \ldots)
  \, ,
\]
where $(\xi_k)$ are independent and identically distributed (iid) random variables,
and $\psi$ satisfies certain regularity assumptions.
Then, Merlev\`ede and Rio~\cite{MR15} obtained the rate $O(\log n)$
for processes of the type 
\[
  X_n = \sum_{k=1}^n \psi (g_k)
  \, ,
\]
where $\{g_k\}_{k \geq 1}$ is a geometrically ergodic Markov chain
and $\psi$ is bounded.

\medskip

In this paper we suppose that $f$ is a 
\emph{nonuniformly expanding map} in the sense of
Young~\cite{Y99}, and $\mu$ is its unique physical
(Sinai-Ruelle-Bowen) invariant measure.
This covers, for example:
\begin{itemize}
  \item uniformly expanding maps such as the doubling map,
    Gauss continued fraction map, $\beta$-shifts and Gibbs-Markov maps with onto branches;
  \item maps with critical points or indifferent fixed points,
    such as intermittent (Pomeau-Manneville) maps~\cite{LSV99},
    logistic maps with Collet-Eckmann parameters~\cite{BLS03,Y99}
    or Alves-Viana maps~\cite{G06};
  \item factors of nonuniformly hyperbolic maps,
    e.g.\ the collision map for dispersing billiards
    or H\'enon map, which are instrumental
    for proving limit theorems including the ASIP~\cite{MN05}.
\end{itemize}
We state the technical assumptions on the maps to which our results apply
in Subsection~\ref{section:NUE}.

Statistical properties of nonuniformly expanding maps
are often proved using a suitable \emph{inducing scheme.}
One chooses a \emph{base} $Y \subset X$ and a \emph{return time}
$\tau \colon Y \to \bN$ with $f^{\tau(y)}(y) \in Y$ for all $y \in Y$,
so that the \emph{induced map} $f^\tau \colon Y  \to Y$ is particularly nice,
namely full branch Gibbs-Markov.

An inducing scheme comes with a natural ``reference'' probability
measure $m$ on $Y$ (e.g.\ Lebesgue),
and the asymptotics of $m(\tau > n)$ as $n \to \infty$ largely determine
what statistical properties one can prove.
For example, if $f$ is mixing 
(more precisely, $\gcd \{\tau(y) : y \in Y\} = 1$), then the asymptotics of $m(\tau > n)$
give a useful bound on the covariances between the summands in
$S_n(\varphi)$
\cite{G04,KKMe,Y99}:
\begin{itemize}
  \item if $m(\tau \geq n) = O(n^{-\beta})$ with $\beta > 1$,
    then
    \(
      \Cov(\varphi, \varphi \circ f^n)
      = O\bigl(n^{-(\beta-1)}\bigr)
    \);
  \item if $m(\tau \geq n) = O\bigl(\mathrm{e}^{- A n^\gamma}\bigr)$
    with $A > 0$ and $0 < \gamma \leq 1$, then
    \(
      \Cov(\varphi, \varphi \circ f^n)
      = O\bigl(\mathrm{e}^{- B n^\gamma}\bigr)
    \) with some $B > 0$.
\end{itemize}
(Recall that the probability space is $(X, \mu)$.)

Melbourne and Nicol~\cite{MN05} proved the ASIP for $S_n(\varphi)$ 
provided that $\tau \in L^{p}(m)$, $p > 2$. Their rates are of the type
$o(n^{\eps})$, where $\eps \in (1/4,1/2)$ depends on $p$.

\begin{rmk}
  The variance of the Brownian motion in the ASIP for $S_n(\varphi)$ is
  \begin{equation} \label{defc2}
    c^2 = \lim_{n \to \infty } \frac{1}{n} \int |S_n (\varphi )|^2 \, d\mu
    \, .
  \end{equation}
  For nonuniformly expanding maps with $\tau \in L^2(m)$,
  the limit exists by e.g.\ \cite[Cor.~5.5]{KKM}.
\end{rmk}

Korepanov~\cite{K17r} applied the result of~\cite{BLW14} to nonuniformly expanding
dynamical systems, showing the ASIP with rate $o(n^\eps)$ for every $\eps > 0$
under the assumption of exponential tails of the return times,
i.e.\ $m(\tau > n) = O(\mathrm{e}^{-c n})$ with $c > 0$.

The method of~\cite{K17r} only works for exponential decay of return times.
It was improved by the authors of this paper in~\cite{CDKM18},
where we obtained a significantly more general result
which covers maps with polynomial decay of return times, such as
intermittent maps:

\begin{thm}[\cite{CDKM18}]
  \label{thm:NUE:pol}
  Let $f \colon X \to X$ be a nonuniformly expanding map
  (see Section~\ref{section:NUE})
  with the reference measure $m$, return time $\tau$ 
  and physical invariant measure $\mu$.
  Let $\varphi \colon X \to \bR$ be a H\"older continuous
  function with $\int \varphi \, d\mu = 0$.
  Consider the random process
  $S_n(\varphi)  = \sum_{k=0}^{n-1} \varphi \circ f^k$
  on the probability space $(X,\mu)$. 
  \begin{itemize}
    \item[(a)] If $m(\tau > n) = O(n^{-\beta})$ with $\beta > 2$, then
      the process $(S_n(\varphi))_{n \geq 0}$ satisfies the ASIP with
      the rate $o(n^{1/\beta} (\log n)^{1/\beta + \eps})$
      for every $\eps > 0$.
    \item[(b]) If $\int \tau^\beta \, dm < \infty$ with $\beta > 2$, then
      the process $(S_n(\varphi))_{n \geq 0}$ satisfies the ASIP with
      the rate $o(n^{1/\beta})$.
  \end{itemize}
\end{thm}

\begin{rmk}
  The rate in Theorem~\ref{thm:NUE:pol} (a) is essentially optimal:
  $\eps$ cannot be reduced to $0$, for example, for the intermittent maps.
\end{rmk}

When the return times decay faster than polynomially, Theorem~\ref{thm:NUE:pol}
gives the ASIP with  rate $o(n^{\eps})$ for any $\eps > 0$, the same as in~\cite{K17r}.
This, however, is suboptimal in case of exponential and stretched exponential tails.
We fix this in the present paper. Our main result is:

\begin{thm} \label{thm:NUE}
  Let $f \colon X \to X$ be a nonuniformly expanding map
  (see Section~\ref{section:NUE})
  with the reference measure $m$, return time $\tau$ 
  and physical invariant measure $\mu$.
  Let $\varphi \colon X \to \bR$ be a H\"older continuous
  function with $\int \varphi \, d\mu = 0$.
  Consider the random process
  $S_n(\varphi)  = \sum_{k=0}^{n-1} \varphi \circ f^k$
  on the probability space $(X,\mu)$. 
  
  If $m(\tau > n) = O(\mathrm{e}^{-\kappa n^\gamma})$
  with $\kappa > 0$ and $\gamma \in ]0,1]$,
  then the process $(S_n(\varphi))_{n \geq 0}$ satisfies the ASIP with
  the rate $O( (\log n )^{1 + 1/\gamma})$.
\end{thm}

\begin{rmk}
  For the doubling map, Theorem~\ref{thm:NUE} gives the rate
  $O( (\log n)^2 )$, improving the rate $o(n^\eps)$
  for every $\eps > 0$ in~\cite{BLW14}.
  We conjecture that it can be further improved to $O( \log n )$,
  which is known to be optimal for
  processes with independent and identically  distributed increments~\cite{KMT75}
  and additive bounded functionals of 
  geometrically ergodic Markov chains~\cite{MR15}.
  We expect the same for all maps which are nonuniformly expanding
  with exponential decay of return times: our rate $O( (\log n)^2 )$
  should be eventually reduced to $O( \log n )$.
\end{rmk}

\begin{rmk}
  In most examples of nonuniformly expanding maps,
  the return times decay exponentially, except the intermittent maps with
  polynomial decay and Alves-Viana maps where the best available
  estimates are stretched exponential, namely $\mathrm{e}^{-c \sqrt{n}}$,
  see~\cite{BG03,G06}.
  To enhance our portfolio of examples, in Appendix~\ref{sec:example}
  we present a family
  of interval maps with decay $O(\mathrm{e}^{-c n^\gamma})$,
  parametrized by $\gamma \in ]0,1]$.
\end{rmk}

\begin{rmk}
  In a class of nonuniformly hyperbolic dynamical systems, 
  the ASIP can be deduced from the corresponding result on
  a nonuniformly expanding map, with the same rate, using the
  so-called Sinai trick~\cite{MN05}.
  For example, for dispersing billiards and H\'enon maps,
  where the return times have exponential tails,
  we obtain the rate $O( (\log n)^2 )$. 
\end{rmk}

The paper is organized as follows. In Section~\ref{sec:prelim}, we give
a formal definition of the class of nonuniformly expanding maps
to which our results apply. Further, we redefine the random process
$S_n(\varphi)$ on a Markov shift without changing its distribution.
In Section~\ref{sectionproofNUE}, we prove Theorem~\ref{thm:NUE}.
The proof is based on the construction of approximating Brownian motion
from~\cite{BLW14} as in~\cite{CDM17} and~\cite{CDKM18}.
One of the crucial tools is the ASIP for processes
with independent (but not necessarily identically distributed)
increments by Sakhanenko~\cite[Thm. 1]{Sa84}.  Sakhanenko's result
generalizes  KMT's  result (concerning iid random variables having a finite moment generating function in a neighborhood of $0$) to 
non-identically distributed random variables whose distributions satisfy a  condition equivalent to the condition in Bernstein's well-known inequality 
(as shown by  Zaitsev \cite{Zait84}).

\medskip

Throughout the paper, we shall often use the notation $a_{n} \ll b_{n}$ which means
that there exists a universal constant $C$ such that, for all $n \geq1$,
$a_{n} \leq C b_{n}$.

\section{Reduction to a Markov shift}
\label{sec:prelim}

\subsection{Nonuniformly expanding maps}
\label{section:NUE}

Here we state formal assumptions on dynamical systems, to which
our results apply. Briefly, we require that they admit a Young tower,
i.e.\ an inducing scheme with a full branch Gibbs-Markov base map
and certain regularity assumptions.
Often Young towers are difficult to construct,
but they provide a universal framework for proving limit theorems. 
For uniformly expanding maps, including those that are not Markov or conformal, one could verify the general assumptions of Eslami~\cite[Sec.~7.2]{P18}.

Let $X$ be a complete bounded separable metric space with the Borel $\sigma$-algebra.
Suppose that $f \colon X \to X$ is a measurable transformation which
admits an inducing scheme consisting of:
\begin{itemize}
  \item a closed subset $Y$ of $X$ with a \emph{reference} probability measure $m$ on $Y$;
  \item a finite or countable partition $\alpha$ of $Y$ (up to a zero measure set)
    with $m(a) > 0$ for all $a \in \alpha$;
  \item an integrable \emph{return time} function $\tau \colon Y \to \{1,2,\ldots\}$ which is constant
    on each $a \in \alpha$ with value $\tau(a)$ and
    $f^{\tau(a)}(y) \in Y$ for all $y \in a$, $a \in \alpha$.
    (We do not require that $\tau$ is the first return time to $Y$.)
\end{itemize}

Let $F \colon Y \to Y$, $F(y) = f^{\tau(y)}(y)$ be the induced map.
We assume that there are constants $\kappa > 1$, $K > 0$ and $\eta \in (0,1]$
such that for each $a \in \alpha$ and all $x,y \in a$:
\begin{itemize}
  \item $F$ restricts to a (measure-theoretic) bijection from $a$ to $Y$;
  \item $d(F(x), F(y)) \geq \kappa d(x,y)$;
  \item $d(f^k(x), f^k(y)) \leq K d(F(x), F(y))$
    for all $0 \leq k \leq \tau(a)$;
  \item the inverse Jacobian $\zeta_a = \frac{dm}{dm \circ F}$ of
    the restriction $F \colon a \to Y$ satisfies
    \[
      \bigl| \log |\zeta_a(x)| - \log |\zeta_a(y)| \bigr|
      \leq K d(F(x), F(y))^\eta
      .
    \]
\end{itemize}

In addition to the standard assumptions above, we rely on
non-pathological coding of orbits under $F$ allow by the elements
of $\alpha$.
Let $\cA$ be the set of all finite words in the alphabet $\alpha$
and $Y_w= \cap_{k=0}^n F^{-k}(a_k)$ for $w=a_0\cdots a_n \in \cA$.
We require that
\[
  m(Y_w)=m(\bar Y_w)
  \qquad \text{for every } w \in \cA
  .
\]

We say that the map $f$ as above is nonuniformly expanding.
We refer to $F \colon Y \to Y$, $F(x) = f^{\tau(x)}(x)$ as the induced map.
It is standard \cite[Cor. p.~199]{AD}, \cite[Proof of Thm.~1]{Y99} that there is
a unique absolutely continuous $F$-invariant probability measure
$\mu_Y$ on $Y$ with $\frac1c \le d\mu_Y / dm\le c$ for some $ c>0$,
and the corresponding $f$-invariant probability measure $\mu$ on $X$.

We say that the return times of $f$ have:
\begin{itemize}
  \item a weak polynomial moment of order $\beta \geq 1$, if $m(\tau  \geq n) \ll n^{-\beta}$;
  \item a strong polynomial moment of order $\beta \geq 1$, if $\int \tau^\beta \, dm < \infty$;
  \item a subexponential moment of order $\gamma \in ]0,1]$, if $\int {\rm e}^{\delta \tau^{\gamma}} d m < \infty$ for some $\delta >0$.
\end{itemize}

\subsection{Markov shift}
\label{sec:MC}

Following~\cite{CDKM18,K17r}, for nonuniformly expanding dynamical systems,
the random process $S_n(\varphi)$ can be redefined on a Markov shift
without changing its distribution. The structure of the Markov shift is as follows.   

Let \((\cA, \bP_\cA)\) be a countable probability space and
\(h \colon \cA \to \{1,2,\ldots\}\) with $\bE_\cA(h) < \infty$.
Let 
\[
  S = \{(w, \ell) \in \cA \times \bZ : 0 \leq \ell < h(w)\}
  \, .
\]

We construct a stationary Markov chain $g_0, g_1, \ldots$ on $S$ such that if
$g_n = (w, \ell)$ with $\ell < h(w) - 1$, then $g_{n+1} = (w, \ell + 1)$,
while if $g_n = (w, \ell)$ with $\ell = h(w) - 1$, then $g_{n+1} = (w', 0)$,
with $w' \sim \bP_\cA$ independent from $(g_k)_{k \leq n}$.

The stationary measure of our Markov chain we denote by $\nu$.
For $(w, \ell) \in S$,
\begin{equation}
  \label{defnu}
  \nu (w, \ell ) 
  =  \frac{\bP_\cA(\omega)}{\bE_\cA (h)}
  \, .
\end{equation}

It is convenient to represent $(g_n)_{n \geq 0}$ as generated by a sequence of
independent \emph{innovations,} as follows. 
Let $g_0 \sim \nu$ and let \(\eps_1, \eps_2, \ldots\) be a sequence of 
independent (also from $g_0$) and identically distributed random variables
with values in \(\cA\) and distribution \(\bP_\cA\).
Let
\begin{equation}
  \label{defgnviaU}
  g_{n+1} 
  = U(g_n, \eps_{n+1}) \, ,
\end{equation}
where
\begin{equation}
  \label{defUforMC} 
  U((w, \ell), \eps)
  = \begin{cases}
    (w, \ell+1), & \ell < h(w) - 1 \, ,  \\
    (\eps, 0), & \ell = h(w) - 1 \, .
  \end{cases}
\end{equation}

Let $\Omega \subset S^\bN$ be the space of possible trajectories of $(g_n)$
(i.e.\ sequences which correspond to non-zero probability transitions),
and let $\bP_\Omega$ be the corresponding probability measure.

Let $ \lambda >1$. For
$a = (g_0, \ldots, g_n, g_{n+1},  \ldots)$ and
$b = (g_0, \ldots, g_n, g'_{n+1}, \ldots)$
with \(g_{n+1} \neq g'_{n+1}\),
let
\begin{equation} \label{distanceseparation}
  d(a,b) = \lambda^{-\#\{ 1 \leq k \leq n : g_k \in S_0 \}}
  \, ,
\end{equation}
where $S_0 = \{(w, \ell) \in S : \ell = 0\}$.
Then $d$ is a separation metric on $\Omega$, the separation time
counted in terms of returns to $S_0$.

We proved~\cite{CDKM18,K17r} that given a nonuniformly expanding dynamical system
$f \colon X \to X$ with a H\"older continuous observable $\varphi \colon X \to \bR$,
there exists a Markov chain as above
and a H\"older continuous
function $\psi \colon \Omega \to \bR$, such that
\[
  \{\varphi \circ f^n\}_{n \geq 0}
  \stackrel{d}{=}
  \{\psi(g_n, g_{n+1}, \ldots)\}_{n\geq 0}
  \, .
\]
(The equality is in law, and the probability measures are $\mu$ and $\bP_\Omega$
respectively.)
Moreover, $h$ has essentially the same tails as $\tau$:
\begin{itemize}
  \item (weak polynomial moment) if $m(\tau  \geq n) \ll n^{-\beta}$
    with $\beta \geq 1$, then $\bP_\cA(h \geq n) \ll n^{-\beta}$;
  \item (strong polynomial moment) if $\int \tau^\beta \, dm < \infty$
     with $\beta \geq 1$, then $\int h^\beta \, d\bP_\cA < \infty$;
  \item (subexponential moment) if $\int {\rm e}^{\delta \tau^{\gamma}} d m < \infty$
    with $\gamma \in ]0,1]$ and $\delta >0$, then
    $\int {\rm e}^{\delta' h^{\gamma}} d \bP_\cA < \infty$
    with some $\delta' > 0$.
\end{itemize} 
Denote 
\begin{equation}
  \label{defXnwithgn}
  X_n = \psi(g_n, g_{n+1}, \ldots)
  \qquad \text{and} \qquad 
  S_n = \sum_{k=1}^{n} X_k
  \, .
\end{equation}
Thus, the ASIP for $S_n(\varphi)$ is reduced to the ASIP for $S_n $.

\subsection{Meeting time}

Following~\cite[Appendix~A]{CDKM18}, for the purpose of proving the ASIP,
we assume without loss of generality that
$\gcd \{h(w) \colon w \in \cA\} = 1$, that is the Markov chain $(g_n)$ is aperiodic.

Let \(g_0^*\) be a random variable in \(S\) with distribution $\nu$,
independent from \(g_0\) and \((\eps_n)_{n \geq 1}\).
Let \(g_0^*, g_1^*, g_2^*, \ldots\) be a Markov chain given by
\[
  g_{n+1}^* = U(g_n^*, \eps_{n+1})
  \ \text{ for } \ n \geq 0 \, .
\]
Thus the chains \( (g_n)_{n \geq 0}\) and \(( g_n^*)_{n \geq 0}\)
have independent initial states, but share the same innovations.
Define the meeting time:
\begin{equation}
  \label{defofTtilde}
  T = \inf \{n \geq 0 : g_n = g^*_n \} \, .
\end{equation}

Recall that $\bP_\cA(h \geq n) = O(\mathrm{e}^{-c n^\gamma})$. This translates into
a similar bound for $T$:

\begin{lemma} \label{lem:nngh}
  There exists $\delta > 0$ such that 
  $ \bP(T \geq n) = O(\mathrm{e}^{-\delta n^\gamma})$.
\end{lemma}

The proof is omitted since it uses the same argument as in \cite[Lemma 3.1]{CDKM18},
namely the result of Lindvall~\cite{Li79} (see also \cite[Prop.~9.6]{Rio17}).

It is noteworthy to indicate that Lemma \ref{lem:nngh} implies the following control on the covariances:

\begin{lemma}
  \label{lmacovtildeXbis}
  There exists $\delta > 0$ such that
  \(
    |\Cov (X_0, X_n) | 
    =O(\mathrm{e}^{-\delta n^\gamma})
  \).
\end{lemma} 
To prove the lemma above it suffices to  follow the proof of Lemma 3.3 in \cite{CDKM18} and to take into account Lemma \ref{lem:nngh} and Proposition \ref{prop:boundofthedelta1first} below whose proof is postponed to Appendix~\ref{sec:prop:boundofthedelta1first}.

\begin{prop}
  \label{prop:boundofthedelta1first}
  Let $\delta_n \colon \Omega \to \bR$,
  \[
    \delta_n (g_0, g_1, \ldots)
    = \sup \bigl| \psi(g_0, \ldots, g_n, g_{n+1}, g_{n+2}, \ldots) - 
    \psi(g_0, \ldots, g_n, \tg_{n+1}, \tg_{n+2}, \ldots) \bigr|
    \, ,
  \]
  where the supremum is taken over all possible
  $(\tg_{n+1}, \tg_{n+2}, \ldots)$.
  Then there exists $\delta > 0$ such that
  \[
    \bE (\delta_n)
    = O(\mathrm{e}^{-\delta n^\gamma})
    \, .
  \]
\end{prop}

\section{Proof of Theorem \ref{thm:NUE}} \label{sectionproofNUE}

Let $\alpha = 1 + \gamma^{-1}$. Our goal is to
prove the ASIP for the random process $(S_n)$, driven by the stationary Markov chain $(g_n)$,
as defined in Section~\ref{sec:prelim} (see the definition \eqref{defXnwithgn}). Recall also that following~\cite[Appendix~A]{CDKM18}, we can and do assume without loss of generality that  the Markov chain $(g_n)$ is aperiodic.

The variance of the Brownian motion in the ASIP is, necessarily,
\[
  c^2= \lim_{n \to \infty } \frac{1}{n} \int |S_n (\varphi )|^2 \, d\mu =  \lim_{n \to \infty } \frac{\bE (S_n^2)}{n} 
  \, .
\] 
Technically, we prove the following strong approximation:
one can redefine $(S_n)_{n \geq 1}$ without changing its distribution
on a  probability space (possibly richer than $(\Omega, \bP_{\Omega} )$),
on which there exists a sequence $(N_i)_{i \geq 1}$ of iid centered Gaussian r.v.'s with variance 
$c^2 $ such that 
\begin{equation}
  \label{resultSAter}
  \sup_{k \leq n } \Bigl| S_n - \sum_{i=1}^k N_i \Bigr| 
  = O( (\log n)^{ \alpha} )  \quad \text{a.s.} 
\end{equation}

Assume first that $c^2=0$. Note that, for any $\varepsilon >0$, 
 \[
   \sum_{n\geq 1} m \left ( \tau > \varepsilon  (\log n)^{ \alpha}  \right )
   < \infty \, .
  \]
Therefore using the same arguments as those developed in the proof of Corollary 5.5 in \cite{CDKM18}, we can conclude that  Theorem \ref{thm:NUE} holds with $c^2=0$. 

Through the reminder of this section, we assume that $c^2>0$ and use the notation $b_n = \lceil (\log n)/(\log 3) \rceil$
for $n \geq 2$ (so that $b_n$ is the unique integer such that $3^{b_n -1} < n \leq 3^{b_n}$).
Let $\delta$ be the minimum of the constants $\delta$ involved in Lemmas~\ref{lem:nngh} and~\ref{lmacovtildeXbis} and Proposition~\ref{prop:boundofthedelta1first}.
Fix $\kappa > 0$ so that
$\delta (2^{-1} \kappa)^{\gamma}  \geq  \log 3$. 
For $\ell \geq 0$, let
\begin{equation} \label{defml} 
  m_{\ell}= [  \kappa \ell^{1/\gamma}]  \vee 1
  \, .
\end{equation}

\medskip

Following~\cite{BLW14}, the proof of \eqref{resultSAter}
it is divided into several steps.

\subsection{Step 1}

Let
\[
X_{\ell , k} =  \bE_g \big ( \psi (g_k, g_{k+1}, \ldots, g_{k+m_{\ell}},   ({\tilde g}_{i})_{i \geq  k+m_{\ell} +1} )  \big ) \, , 
\]
where $\bE_g$ denotes the conditional expectation given $g=(g_n )_{n \geq 0}$. 
Here $({\tilde g}_{i})_{i \geq  k+m_{\ell} +1} $ is defined as follows: ${\tilde g}_{k+m_{\ell} +1} = U (g_{k+m_{\ell} } , \eps'_{k+m_{\ell} +1})$ and 
${\tilde g}_{i+1} = U (  {\tilde g}_{i}, \eps'_{i+1} ) $ for any $i > k+m_{\ell} $, where $(\eps_i')_{i \geq 1}$ is an independent copy of $(\eps_i)_{i \geq 1}$, independent of $g_0$, and $U$ is given by \eqref{defUforMC}. Note that the $X_{\ell , k}$'s are centered. 
Define
\[
  W_{\ell, i} = \sum_{k=1+3^{\ell-1}}^{ i  +  3^{\ell-1}}  X_k
  \, , \qquad  
  {\overline W}_{\ell, i} = \sum_{k=1+3^{\ell-1}}^{ i  +  3^{\ell-1}}  {X}_{\ell , k }
  \qquad \text{and} \qquad
  W'_{\ell, i} = W_{\ell, i} - {\overline W}_{\ell, i}  \, .
 \]
The fist step is to prove
\begin{lemma}
\begin{equation} \label{stepone}
\sum_{\ell=1}^{ b_n -1}  { W}'_{\ell, 3^{\ell} - 3^{\ell-1} }   +  { W}'_{b_n,  n - 3^{b_n-1}  }   = O( (\log n)^{ \alpha} )     \quad  \text{a.s.}
\end{equation}
\end{lemma}

\begin{proof}
  By Proposition~\ref{prop:boundofthedelta1first},
  \[
  \Bigl\| \max_{1 \leq i \leq 3^\ell - 3^{\ell-1}} \bigl|W'_{k, \ell} \bigr| \Bigr\|_1
  \leq  \sum_{k=1+3^{\ell-1}}^{ 3^{\ell}} \Vert X_{k}  - X_{\ell, k }  \Vert_1 
  \ll 3^\ell \exp ( -\delta m_\ell^{\gamma} )
  \leq 3^\ell \exp (-  \delta \times (2^{-1} \kappa )^{\gamma}  \ell )    \, .
  \]
  Using $\delta \times (2^{-1} \kappa )^{\gamma}  \geq  \log 3 $,
  \[
    \sum_{\ell \geq 1}  \ell^{-\alpha} \Bigl\| \max_{1 \leq i \leq 3^\ell - 3^{\ell-1}} 
    \bigl|W'_{k, \ell} \bigr| \Bigr\|_1 
    < \infty \, .
  \]
  Now~\eqref{stepone} follows from the Kronecker's lemma.
\end{proof}

\subsection{Step 2.}

For $k \geq m_{\ell }+1$, let
\begin{equation} \label{deftildeXkl}
  {\tilde X}_{\ell , k} = \bE ( X_{\ell , k}  | \eps_{k - m_{\ell}} , \ldots,  \eps_{k + m_{\ell}}  )
  \, .
\end{equation}
Set $\ell_0 = \inf \{ \ell \geq 1 \, : \, 3^{\ell -1} \geq \kappa \ell^{1/\gamma} \}$. For $\ell \geq \ell_0$, define  
\[
  { \widetilde W}_{\ell, i} = \sum_{k=1+3^{\ell-1}}^{ i  +  3^{\ell-1}}  {\tilde X}_{\ell, k }
  \qquad \text{and} \qquad
 { W}''_{\ell, i}  =   { \overline W}_{\ell, i} -{ \widetilde W}_{\ell, i }
 \, .
\]
In the second step we prove
\begin{lemma}
\begin{equation} \label{steptwo}
\sum_{\ell=\ell_0}^{ b_n -1}  { W}''_{\ell, 3^{\ell} - 3^{\ell-1} }  +  { W}''_{b_n,  n - 3^{b_n-1}  } = O( (\log n)^{ \alpha} )  \quad  \text{a.s.}
\end{equation}
\end{lemma}

\begin{proof}
  The result follows from the Kronecker's lemma, once we show that
  \begin{equation} \label{krosteptwo}
    \sum_{\ell \geq \ell_0} \ell^{-\alpha} \sum_{k=3^{\ell-1}+1}^{3^\ell}
    \bigl\| X_{ \ell , k }  - {\tilde X}_{\ell , k} \bigr\|_1
    < \infty \, .   
  \end{equation}
  By the estimate \cite[(4.10)]{CDKM18}, 
  \[
    \sum_{\ell \geq \ell_0} \ell^{ - \alpha} 
    \sum_{k= 3^{\ell-1}+1}^{3^\ell}
    \bigl\| X_{ \ell , k }  - {\tilde X}_{\ell , k} \bigr\|_1
    \leq 2 |\psi|_{\infty}
    \sum_{\ell \geq \ell_0} \frac{3^{\ell}}{\ell^{\alpha}} \bP (T \geq  m_{\ell})
    \, .
  \]
  By Lemma~\ref{lem:nngh}, $\bP(T \geq n) = O(\exp(-\delta n^\gamma))$.
  Since $m_\ell \geq \kappa \ell^{1/\gamma}/2$ and $\delta (2^{-1} \kappa )^{\gamma}  \geq  \log 3 $,
  the bounds~\eqref{krosteptwo} and \eqref{steptwo} follow.
\end{proof}

\subsection{Step 3. (\textit{Conditional Gaussian approximation})}

Set 
\[
 {\tilde S}_n= \sum_{\ell=\ell_0}^{b_n-1} {\widetilde W}_{\ell, 3^\ell - 3^{\ell-1}} +{\widetilde W}_{b_n,  n - 3^{b_n-1}  }  \, . 
\]
Let $K_0 = \inf \{ k \geq 1: m_k \leq 4^{-1} 3^{k-2} \}$. For
$\ell \geq K_0$, let 
\[
q_\ell = [ 3^{\ell-2} / m_\ell] -2 \, .
\]
Note that $q_\ell \rightarrow \infty$, as $\ell \rightarrow \infty$ and $q_\ell \geq 2$ whenever $\ell \geq K_0$. For $\ell \geq K_0$ and $j=1, \ldots, q_\ell$, set 
\begin{equation} \label{defBkj}
B_{\ell,j} =  \sum_{i=1+ (6j-1) m_\ell }^{ (6 j +5) m_\ell} {\tilde X}_{\ell,i +m_\ell +3^{\ell-1}} \, . 
\end{equation}
Let $B_{\ell,j} = 0$ if $\ell < K_0$.  In what follows, we assume that  $n \geq N_0= 3^{K_0}$. Define
\begin{equation} \label{defSndiamond}
S_n^{\diamond} =  \sum_{\ell=K_0}^{b_n-1} \sum_{j=1}^{q_\ell}B_{\ell,j}  + \sum_{j=1}^{\tau_n}B_{b_n,j}  \, , \, \text{ where } \tau_n = \Big [  \frac{ n-3^{b_n-1}}{6m_{b_n}}\Big ] -2 \, .
\end{equation}
Note that $\tau_n \leq q_{b_n}$. Moreover, since $\Vert {\tilde X}_{k,i} \Vert_{\infty} \leq |\psi|_{\infty} $ a.s., we infer  that there exists a positive constant $C$ not depending on $n$ such that 
\begin{equation} \label{p3KMT}
\max_{N_0 \leq i \leq n} \big |  {\tilde S}_i  - S_i^{\diamond} \big |  \leq C      \sum_{k=1}^{b_n}  m_k  =O( ( \log n )^{ \alpha}) \   \text{ a.s.} 
\end{equation}
Taking into account  \eqref{stepone}, \eqref{steptwo} and \eqref{p3KMT}, we see that (using for instance Lemma 4.1 of Berkes-Liu-Wu \cite{BLW14}) \eqref{resultSAter} is reduced to prove that one can redefine $( S^{\diamond}_{n})_{n \geq 1}$ without changing its distribution on a (richer) probability space on which 
there exists iid random variables $(N_i)_{i \geq 1}$ with common distribution ${\mathcal N} (0, c^2)$, such that,
\begin{equation} \label{KMTSRbis}
S_n^{\diamond} - \sum_{i=1}^n N_i = O(( \log n)^{\alpha}) \, \text{ a.s.}
\end{equation}
where we recall that $S_n^{\diamond}$ is defined in \eqref{defSndiamond}. 
To prove this strong approximation result, we proceed as in steps 3.2 and 3.3 of the proof of Theorem 2.1 in Berkes-Liu-Wu \cite{BLW14}. Their step 3.2 consists in  showing a conditional Gaussian approximation that is the object of our step 3. This requires several preliminary notations that we recall below for an easy understanding of the proof. With this aim, note first that for any 
integer $i \in [3^{\ell-1}+1, 3^\ell]$ with $\ell \geq K_0$, we can write
\[
{\tilde X}_{\ell,i} =G_{\ell} (\varepsilon_{i-m_\ell}, \ldots, \varepsilon_{i + m_{\ell}}) \, ,
\]
where $G_{\ell}$ is a bounded measurable function. So ${\tilde X}_{\ell , i}$ is a bounded measurable function of  $(  \eps_{i - m_{\ell}} , \ldots,  \eps_{i + m_{\ell}}  )$.

Define, for $j \geq 1$,
\[
{\mathcal J}_{\ell,j} = \{ 3^{\ell-1} +  (6j-1) m_\ell +  k , k=1,2, \ldots, 2m_\ell\}  \, , 
\]
\[
{\bm \eta}_{\ell,j} = (\eps_{i}, i \in {\mathcal J}_{\ell,j}  )  \,  \text{ and } \, {\bm \eta}= ({\bm \eta}_{\ell,j}, j =1, \ldots, q_\ell+1)_{\ell=K_0}^{\infty} \, .
\]
Note that 
\begin{multline*}
B_{\ell,j} =  \sum_{i=1+ (6j-1) m_\ell }^{ (6 j +1) m_\ell} {\tilde X}_{\ell,i +m_\ell +3^{\ell-1}}  +  \sum_{i=1+ (6j +1) m_\ell }^{ (6 j +3) m_\ell} {\tilde X}_{\ell,i +m_\ell +3^{\ell-1}}  +  \sum_{i=1+ (6j+3) m_\ell }^{ (6 j +5) m_\ell} {\tilde X}_{\ell,i +m_\ell +3^{\ell-1}}   \\ := H_\ell \big ( {\bm \eta}_{\ell,j}, \{\eps_{i + 3^{\ell-1}}\}_{ 1+(6j +1) m_\ell  \leq i  \leq (6j +5) m_\ell} , {\bm \eta}_{\ell , j+1} \big ) \, .
\end{multline*}
Let now $( u_{\ell}, \ell \in \N )$ be elements of ${\mathcal A}$ and set ${\bf u} = ({\bf u }_{k,j}, j =1, \ldots, q_k+1)_{k=K_0}^{\infty}$   
 where for any $j=1, \ldots, q_k+1$, ${\bf u }_{k,j} = (u_{\ell}, \ell \in {\mathcal J}_{k,j}  )$. The idea is to use the fact that, on the set $\{{\bm \eta} = {\bf u} \}$,  $(B_{\ell,j} ( {\bf u} ))_{j=1, \ldots, q_\ell}$ are  independent between them. With this aim, define the following random functions: for $j \geq 1$, 
\[
F^{(1)}_{k,j} ({\bf u }_{k,j})  =  \sum_{i=1+ (6j-1) m_k }^{ (6 j +1) m_k} G_k \big ( u_{i+3^{k-1}}, \ldots, u_{(6 j +1) m_k+3^{k-1}}, \varepsilon_{(6 j +1) m_k +1+3^{k-1}}, \ldots, \varepsilon_{ i+2m_k+3^{k-1}} \big ) \, , 
\]
\[
F^{(2)}_{k,j}   =  \sum_{i=1+ (6j +1) m_k }^{ (6 j +3) m_k}  G_k \big ( \varepsilon_{i+3^{k-1}}, \ldots, \varepsilon_{(6j+1)m_k +3^{k-1}}, \varepsilon_{(6j+1)m_k +1+3^{k-1}}, \ldots, \varepsilon_{ i+2m_k+3^{k-1}} \big ) \, , 
\]
and
\[
F^{(3)}_{k,j} ({\bf u }_{k,j +1})   =\sum_{i=1+ (6j+3) m_k}^{ (6 j +5) m_k} G_k \big ( \varepsilon_{i+3^{k-1}}, \ldots, \varepsilon_{(6j+5)m_k +3^{k-1}}, u_{(6j+5)m_k +1+3^{k-1}}, \ldots, u_{ i+2m_k+3^{k-1}} \big ) \, .
\]
Note that $F^{(2)}_{k,j} $ is centered but not the two others processes defined above. Their  mean functions are denoted by
\[
{\Lambda }_{k,1}({\bf u }_{k,j}) := \E F^{(1)}_{k,j}  ({\bf u }_{k,j}) \, \text{ and } \, {\Lambda }_{k,3}({\bf u }_{k,j+1}) = \E F^{(3)}_{k,j}  ({\bf u }_{k,j+1}) \, .
\]
Note that, for any $j=1, \ldots, q_k+1$,   we have
\[
B_{k,j} = F^{(1)}_{k,j}  ({\bm \eta}_{k,j})  +  F^{(2)}_{k,j} +F^{(3)}_{k,j}  ({\bm \eta}_{k,j +1})  \, .
\]
Let us now introduce the centered process
\[
Y_{k,j} ( {\bf u }_{k,j}, {\bf u }_{k,j+1}) =   F^{(0,1)}_{k,j} ({\bf u }_{k,j})   +  F^{(2)}_{k,j} +   F^{(0,3)}_{k,j} ({\bf u }_{k, j +1})   \, ,
\]
where 
\[
 F^{(0,1)}_{k,j} ({\bf u }_{k,j})  =  F^{(1)}_{k,j} ({\bf u }_{k,j})  -  {\Lambda }_{k,1}({\bf u }_{k,j}) \, \text{ and } \,   F^{(0,3)}_{k,j} ({\bf u }_{k,j +1}) =F^{(3)}_{k,j}  ({\bf u }_{k,j +1}) - {\Lambda }_{k,3}({\bf u }_{k,j +1})  \, .
\]
Note that $Y_{k,j} ( {\bf u }_{k,j}, {\bf u }_{k,j+1})$, $j=1, \ldots, q_k$, $k \geq K_0$ are then mean zero independent random variables with variance function denoted by 
\[
V_{k} ( {\bf u }_{k,j}, {\bf u}_{k,j+1})  :=\Vert Y_{k,j} ( {\bf u }_{k,j}, {\bf u }_{k,j+1})  \Vert_2^2 \, . 
\]
Define then
\begin{equation} \label{defHna}
b_n ({\bf u}) = \sum_{k = K_0}^{b_n-1} \sum_{j=1}^{q_k} Y_{k,j} ( {\bf u }_{k,j}, {\bf u}_{k,j+1}) +  \sum_{j=1}^{\tau_n} Y_{b_n,j} ( {\bf u }_ {b_n,j}, {\bf u }_ {b_n,j+1}) \, .
\end{equation}
Using Theorem 1 in Sakhanenko \cite{Sa84} as it will be done later, it is possible to infer that we can strongly approximate $b_n ({\bf u}) $ by a Brownian motion and that the error in the strong approximation is of the right order. However, the variance of the approximating Brownian motion will be the  variance of $b_n ({\bf u}) $ that is
\[
{\rm var} (b_n ({\bf u}) ):= Q_n({\bf u}) = \sum_{k = K_0}^{b_n-1} \sum_{j=1}^{q_k} V_{k} ( {\bf u }_{k,j}, {\bf u }_{k,j+1})  +  \sum_{j=1}^{\tau_n} V_{b_n} ( {\bf u }_ {b_n,j}, {\bf u}_ {b_n,j+1}) \, .
\]
Since, for $k$ fixed, the random variables $(V_{k} ( {\bm \eta }_{k,j}, {\bm \eta }_{k,j+1}))_{j \geq 1}$ are not independent, this creates problems to proceed to the unconditional Gaussian approximation as done in Step 3.3 in Berkes-Liu-Wu 
\cite{BLW14}. This is the reason why Berkes-Liu-Wu 
\cite{BLW14} have introduced another process ${\Gamma}_n ({\bf u})$ and rather than approximating $b_n ({\bf u})$, they approximate the process
\begin{equation} \label{defHn0a}
H^{\circ}_n ({\bf u}) := b_n ({\bf u}) + {\Gamma}_n ({\bf u})  \, . 
\end{equation}
The process ${\Gamma}_n ({\bf u}) $ is defined as follows:
\begin{equation} \label{defGammana}
{\Gamma}_n ({\bf u}) = \sum_{k = K_0}^{b_n-1} L^{1/2}_k ({\bf u }_{k,1 }) \zeta_k +  L^{1/2}_{b_n} ({\bf u }_{b_n,1 }) \zeta_{b_n} \, ,
\end{equation}
where $(\zeta_{\ell})_{\ell \in {\mathbb Z}}$ is a sequence of iid standard normal random variables which is independent of 
$(\varepsilon_{\ell})_{\ell \in {\mathbb Z}}$,   
\[
L_k ({\bf u }_{k,j}) = \Vert    F^{(2)}_{k,j} +   F^{(0,3)}_{k,j} ({\bf u }_{k,j})  \Vert_2^2 \, , \, j=1, \ldots, q_k +1 \, , 
\]
with, for any $j=1, \ldots, q_k+1$,  
\[
 F^{(0,3)}_{k,j} ({\bf u }_{k,j})  =F^{(3)}_{k,j}  ({\bf u }_{k,j})    - {\Lambda }_{k,3}({\bf u }_{k,j}) \, , 
\]
\[
F^{(3)}_{k,j}  ({\bf u }_{k,j})   =\sum_{i=1+ (6j+3) m_k}^{ (6 j +5) m_k} G_k \big ( \varepsilon_{i+3^{k-1}}, \ldots, \varepsilon_{(6j+5)m_k +3^{k-1}}, u_{(6j-1)m_k +1+3^{k-1}}, \ldots, u_{ i-4m_k+3^{k-1}} \big ) \, ,
\]
and
\[ {\Lambda }_{k,3}({\bf u }_{k,j}) = \E F^{(3)}_{k,j}  ({\bf u }_{k,j}) \, .
\]
Note now that the variance of $H^{\circ}_n ({\bf u})$ is
\[
 Q^{\circ}_n({\bf u}) = \sum_{k = K_0}^{b_n-1} \Big  \{ \sum_{j=1}^{q_k} V_{k} ( {\bf u}_{k,j}, {\bf u}_{k,j+1})  + L_k ({\bf u}_{k,1 })  \Big \}+  \sum_{j=1}^{\tau_n} V_{b_n} ( {\bf u}_ {b_n,j}, {\bf u}_ {b_n,j+1}) + L_{b_n} ({\bf u}_{b_n,1 })\, .
\]
But denoting by
\[
V^{0}_{k} ( {\bf u}_{k,j})  :=\Vert  F^{(0,1)}_{k,j} ({\bf u}_{k,j})   +  F^{(2)}_{k,j} +   F^{(0,3)}_{k,j} ({\bf u}_{k,j})  \Vert_2^2 \, , 
\]
the following equality holds: for any positive integer $t$,
\[
L_k ({\bf u}_{k,1})  + \sum_{j=1}^t V_{k} ( {\bf u}_{k,j}, {\bf u}_{k,j+1})  =  \sum_{j=1}^t V^{0}_{k} ( {\bf u}_{k,j}) + L_k ({\bf u}_{k, t+1 }) \, . 
\]
Therefore, the variance of $H^{\circ}_n ({\bf u})$ can be rewritten as:
\begin{equation} \label{varianceofHN0a}
Q^{\circ}_n ({\bf u}) =  \sum_{k = K_0}^{b_n-1}  \Big \{ \sum_{j=1}^{q_k} V^{0}_{k} ( {\bf u}_{k,j}) + L_k ({\bf u}_{k,q_k+1  })   \Big \}   +   \sum_{j=1}^{\tau_n}   \big \{ V^{0}_{b_n} ( {\bf u}_ {b_n,j}) + L_{b_n} ({\bf u}_{b_n, \tau_n +1})   \big \}  \, .
\end{equation}
Since, the random variables $V^{0}_{k} ( {\bm \eta }_{k,j}) $, $j=1, \ldots, q_k, k \geq K_0$, are  independent, it will be then possible to proceed to an unconditional Gaussian approximation (see our step 4). As in Berkes-Liu-Wu \cite{BLW14}, for notational convenience in what follows, for $j=0$, we let  $Y_{k,0}  ( {\bf u}_{k,0}, {\bf u}_{k,1}) := L^{1/2}_k ({\bf u}_{k,1})\zeta_k$. With all the notations above, it follows that
\[
H^{\circ}_n ({\bf u}):= b_n ({\bf u})  +
{\Gamma}_n ({\bf u})  =  \sum_{k = K_0}^{b_n-1} \sum_{j=0}^{q_k} Y_{k,j} ( {\bf u}_{k,j}, {\bf u}_{k,j+1}) +  \sum_{j=0}^{\tau_n} Y_{b_n,j} ( {\bf u}_ {b_n,j}, {\bf u}_ {b_n,j+1})  \, .
\]

\smallskip

The step 3.2 in Berkes-Liu-Wu \cite{BLW14} consists in applying Theorem 1 in Sakhanenko \cite{Sa06}. We shall rather use Theorem 1 in  Sakhanenko \cite{Sa84} (for an easy reference see Theorem A in \cite{Sh95}). With this aim, we  set, for any $k \geq  K_0$  and any $j \geq 0$, 
\[
\zeta_{k,j} = k^{-1/\gamma} Y_{k,j} ( {\bf u}_{k,j}, {\bf u}_{k,j+1})  \, .
\]
Note now that for any $k \geq  K_0$ and any $j \geq 1$,  
\[
\Vert Y_{k,j} ( {\bf u}_{k,j}, {\bf u}_{k,j+1}) \Vert_{\infty} \leq  10  |\psi |_{\infty} m_k  \leq  \kappa_1 k^{1 / \gamma} \  , 
\]
(where $\kappa_1=10 \kappa  |\psi |_{\infty} $) implying that, for any $t >0$, 
\[
t \E \big ( | \zeta_{k,j}  |^3 \mathrm{e}^{t |\zeta_{k,j}  |} \big ) \leq  t \kappa_1   \mathrm{e}^{t \kappa_1 }  \E \big ( |\zeta_{k,j} |^2 \big ) \, .
\]
So, if  $0 < t \leq 1/(2 \kappa_1  ) $ (implying  $t \kappa_1 \mathrm{e}^{t \kappa_1 } \leq 1$), it follows that, for any $k \geq  K_0$ and any $j \geq 1$, 
\[
t \E \big ( | \zeta_{k,j}  |^3 \mathrm{e}^{t |\zeta_{k,j}  |} \big )  \leq  \E \big (|\zeta_{k,j} |^2 \big ) \, .
\]
On another hand 
\[
\vert  L^{1/2}_k ({\bf u}_{k,1})  \vert \leq 6  |\psi |_{\infty} m_k   \leq \kappa_1 k^{1 / \gamma}  \, .
\]
Moreover, for any positive integer $\sigma$, any positive $t$ and any standard Gaussian r.v. $Z$, 
\begin{equation}  \label{trivialgaussian}
\E \big ( |Z|^3 \mathrm{e}^{ t \sigma |Z|} \big )  \leq 2 \mathrm{e}^{t^2 \sigma^2 /2} \Big ( \frac{ ( 2 + (\sigma t)^2  ) \mathrm{e}^{-t^2 \sigma^2 /2}}{\sqrt{2 \pi}}  + 3 \sigma t + (\sigma t )^3  \Big ) : = g(\sigma t ) \, .
\end{equation} 
Applying the  inequality above with  $ \sigma = \kappa_1 $, it follows that,  for any $k \geq  K_0$ and any $ t >0$, 
\[
t \E  \big ( |\zeta_{k,0} |^3 \mathrm{e}^{t |\zeta_{k,0} |} \big ) 
\leq \vert k^{- 1/ \gamma} L^{1/2}_k ({\bf u}_{k,1})  \vert^2  ( \kappa_1  t )   g( \kappa_1 t  ) := \E \big ( |\zeta_{k,0} |^2 \big )   ( \kappa_1  t )   g( \kappa_1 t  )   \, . 
\]
Since there exists $\kappa_2$ such that for any positive $t $ such that $t \leq  1/\kappa_2 $,  we have $( \kappa_1  t )   g( \kappa_1  t  ) \leq 1$, we get that for any  $t \leq  1/\kappa_2$ and any $k \geq K_0$, 
\[
t \E  \big ( |\zeta_{k,0} |^3 \mathrm{e}^{t |\zeta_{k,0} |} \big ) 
\leq   \E \big ( |\zeta_{k,0} |^2 \big )  \, . 
\]
So, overall, for $  t := K =  (  \max (2 \kappa_1 , \kappa_2 ) )^{-1}$, we get that, for any $k \geq  K_0$ and any $j \geq 0$,
\[
t \E \big ( | \zeta_{k,j}  |^3 \mathrm{e}^{t |\zeta_{k,j}  |} \big ) \leq  \E \big ( |\zeta_{k,j} |^2 \big ) \, .
\]
Using Theorem 1 in Sakhanenko \cite{Sa84}, it follows that there exists a probability space $(\Omega_{{\bf u}}, {\mathcal  A}_{{\bf u}} , {\mathbb P}_{{\bf u}} )$ on which we can define random variables $R^{{\bf u}}_{k,j}$ such that 
\[
(R^{{\bf u}}_{k,j})_{0 \leq j \leq q_k, k \geq K_0} =^{\mathcal D} (Y_{k,j} ( {\bf u}_{k,j}, {\bf u}_{k,j+1})  )_{0 \leq j \leq q_k, k \geq K_0} \, , 
\]
and a sequence of independent normal random variables $(N^{{\bf u}}_{k,j}  )_{0 \leq j \leq q_k, k \geq K_0}$ with mean zero and ${\rm Var } (N_{k,j}^{{\bf u}} ) =  k^{-2/\gamma}  {\rm Var } ( Y_{k,j} ( {\bf u}_{k,j}, {\bf u}_{k,j+1}) )$ in such a way  that, for any positive integer $\ell$ and any positive real $x$,
\begin{multline} \label{appliSakhaappli1}
 {\mathbb P}_{{\bf u}}  \big (\max_{ N_0 \leq i \leq 3^{\ell } } \big |  \sum_{k = K_0}^{h_i-1} \sum_{j=0}^{q_k} R^{{\bf u}}_{k,j}+  \sum_{j=0}^{\tau_i} R^{{\bf u}}_{h_i,j} - 
  \sum_{k = K_0}^{h_i-1} \sum_{j=0}^{q_k} k^{1/\gamma} N^{{\bf u}}_{k,j} -   \sum_{j=0}^{\tau_i}  h_i^{1/\gamma} N^{{\bf u}}_{h_i,j} 
 \big | \geq x \big ) \\
 \leq  {\mathbb P}_{{\bf u}}  \big ( 2  \ell ^{1 / \gamma } \max_{ N_0 \leq i \leq 3^{\ell } } \big |  \sum_{k = K_0}^{h_i-1} \sum_{j=0}^{q_k} k^{-1/\gamma} R^{{\bf u}}_{k,j}+  \sum_{j=0}^{\tau_i} h_i^{-1/\gamma}R^{{\bf u}}_{h_i,j} - 
  \sum_{k = K_0}^{h_i-1} \sum_{j=0}^{q_k}  N^{{\bf u}}_{k,j} -   \sum_{j=0}^{\tau_i}  N^{{\bf u}}_{h_i,j} 
 \big | \geq x  \big ) \\
 \leq  \Big ( 1 +  K  \sum_{k = K_0}^{\ell}  \sum_{j=0}^{q_k}  k^{-2/\gamma} \E ( Y^2_{k,j} ( {\bf u}_{k,j}, {\bf u}_{k,j+1})  \Big )  \exp \big (  - A  K x \ell ^{-1 / \gamma } /2  \big )  \, , 
\end{multline}
where $A$ is a universal constant and we recall that $K =  (  \max (2 \kappa_1 , \kappa_2 ) )^{-1}$.  Note that  the first inequality in the inequations above follows from an application of Lemma 2.1 in Shao \cite{Sh95}  which states that if $\{u_n, n \geq 1\}$ is a non-decreasing sequence of positive numbers and if $\{ \zeta_n , n \geq 1 \}$ is a sequence of random variables, then for each $n \geq 1 $, 
\begin{equation} \label{AbelShao}
 \Big |  \sum_{i =1}^n  u_i \zeta_i \Big | \leq 2 u_n  \max_{i \leq n}  \Big |  \sum_{j =1}^i  \zeta_j \Big | \, .  
\end{equation}
Note now that 
\[
  \sum_{k = K_0}^{h_i-1} \sum_{j=0}^{q_k} k^{2/\gamma} {\rm Var } ( N^{{\bf u}}_{k,j} ) +   \sum_{j=0}^{\tau_i}  h_i^{2/\gamma} {\rm Var } (N^{{\bf u}}_{h_i,j} ) = Q^{\circ}_i ({\bf u}) \, , 
\]
where  $ Q^{\circ}_i ({\bf u})$ is   defined in \eqref{varianceofHN0a}. Hence it follows that there is a Brownian motion ${\mathbb B}_{{\bf u}}$ such that for any positive integer $\ell$ and  any positive real $x$, 
\begin{multline} \label{appliSakha}
 {\mathbb P}_{{\bf u}}  \big (\max_{ N_0 \leq i \leq 3^{\ell } } \big |  \sum_{k = K_0}^{h_i-1} \sum_{j=0}^{q_k} R^{{\bf u}}_{k,j}+  \sum_{j=0}^{\tau_i} R^{{\bf u}}_{h_i,j} - {\mathbb B}_{{\bf u}} (Q_i^0 ({\bf u}))  \big | \geq x \big ) \\  \leq  \Big ( 1 +  K  \Psi_\ell ({\bf u} )   \Big ) \exp \big (  - A  K x \ell ^{-1 / \gamma } /2  \big )  \, ,
\end{multline}
where
\[
\Psi_\ell ({\bf u} )  =  \sum_{k = K_0}^{\ell} \sum_{j=0}^{q_k}  k^{-2/\gamma}  \E ( Y^2_{k,j} ( {\bf u}_{k,j}, {\bf u}_{k,j+1}) )  \, .  
\]
Note that 
\[
\E (\Psi_\ell ({\bm \eta} ) ) \ll  \sum_{k = K_0}^{\ell}  k^{-2/\gamma}  q_k  \Vert {\widetilde W}_{k, m_k}   \Vert_2^2 \ll 3^{\ell} \, .
\]
By taking $x = C \ell^{1 + 1 / \gamma }$ in \eqref{appliSakha} with $C= (4 \log 3) /(AK)$, we conclude via the Borel-Cantelli lemma,  that as $n \rightarrow \infty$, 
\begin{equation} \label{conclusionp9}
\max_{ i \leq n  }  \big |  \sum_{k = K_0}^{h_i-1} \sum_{j=0}^{q_k} R^{{\bm \eta}}_{k,j}+  \sum_{j=0}^{\tau_i} R^{{\bm \eta}}_{h_i,j} - {\mathbb B}_{{\bm \eta}} (Q^{\circ}_i ({\bm \eta}))  \big | = O (  ( \log n )^{\alpha})   \, \text{ a.s.}
\end{equation}
This ends the step 3 of our proof. 
\subsection{Step 4. (\textit{Unconditional Gaussian approximation})}

Starting from the conditional Gaussian approximation \eqref{conclusionp9}, we shall prove here that there exists a Brownian motion ${\mathbb B}$ such that 
\begin{equation} \label{(3.48)}
 \max_{1 \leq i \leq n} \Big | S_i^{\diamond} -{\mathbb B}(\sigma_i^2) \Big | =O( (\log n )^{ \alpha}  ) \, \text{ a.s. }
\end{equation}
where 
\begin{equation} \label{defsigmansquare}
\sigma_n^2:=  \sum_{k = K_0}^{b_n-1} q_k \Vert A_{k,1} \Vert_2^2+ \tau_n \Vert A_{b_n,1} \Vert_2^2 \, .
\end{equation}
With this aim, we shall use arguments developed in the step  3.3 in Berkes-Liu-Wu \cite{BLW14} with some modifications.  This step consists first in showing that we can decompose the  Brownian motion ${\mathbb B}_{{\bf u}}$, constructed at Step 3, as
\begin{equation} \label{decompo1brownien}  
{\mathbb B}_{{\bf u}} (Q^{\circ}_n ({\bf u})) = {{\overline w}}_n  ({\bf u}) + \varphi_n  ({\bf u}) \, , 
\end{equation}
where 
\begin{equation} \label{neglivarphin}
\max_{ i \leq n  }   |  \varphi_i  ({\bm \eta}) | = O( (\log n )^{\alpha} )  \, \text{ a.s.}
\end{equation}
and that 
\begin{equation} \label{samelawphiwi}
(\Phi_i)_{i \geq N_0} =^{{\mathcal D}} ({{\overline w}}_i  ({\bm \eta}) )_{i \geq N_0} \, , 
\end{equation}
where
\begin{equation} \label{defPhin}
\Phi_n =  \sum_{k = K_0}^{b_n-1}  \sum_{j=1}^{q_k} (V^{0}_{k} ( {\bm \eta}_{k,j}) )^{1/2} Z^{\star}_{k,j}   +   \sum_{j=1}^{\tau_n}   \  (V^{0}_{b_n} (  {\bm \eta}_ {b_n,j})  )^{1/2} Z^{\star}_{b_n,j} \, ,
\end{equation}
with $ Z^{\star}_{k,j} $, $k,j \in {\mathbb Z}$, independent iid standard normal random variables independent of $(\varepsilon_i)_{i \in {\mathbb Z}}$, and 
\[
\varphi_n  ({\bf u}) = \sum_{k = K_0}^{b_n-1} L^{1/2}_k ({\bf u}_{k,q_k +1 }) {\mathcal G}^{{\bf u}}_{k,1+q_k} +  L^{1/2}_{b_n} ({\bf u}_{b_n,\tau_n+1 }) {\mathcal G}^{{\bf u}}_{b_n, 1 + \tau_n} \, ,
\]
where $( {\mathcal G}^{{\bf u}}_{k,j} )_{k,j}$ are standard normal random variables (that can be possibly dependent) but which are independent of ${\bm \eta}$. 

Note that $\tau_n \leq q_{b_n}$. Hence to prove that $|L^{1/2}_{b_n} ( {\bm \eta}_{b_n,\tau_n+1 }) {\mathcal G}^{ {\bm \eta}}_{b_n, 1 + \tau_n} | =  O( (\log n )^{ \alpha} ) $ a.s., as $n \rightarrow \infty$ (where 
we recall that 
$ \alpha =1+\frac{1}{\gamma} $),  it is enough to show that
\begin{equation} \label{(3.40)} 
\sum_{k \geq K_0} \p \Big (  \max_{1 \leq j \leq q_k } \big |  L^{1/2}_k ({\bm \eta}_{k,j+1 }) {\mathcal G}^{{\bm \eta}}_{k,1+j}  \big |  >  C k^{ 1 + 1/ \gamma}   \Big ) < \infty \, ,
\end{equation}
Using  Markov inequality and the independence between $( {\mathcal G}^{{\bf u}}_{k,j} )_{k,j}$ and ${\bm \eta}$, the fact that  \[
 L_k ({\bf u}_{k, j+1  }  ) \leq (6 m_k |\psi|_{\infty})^2  \leq  \kappa_3 k^{2/ \gamma}\, , 
\]
and that for $N \sim {\mathcal N} (0,1) $, for any $x>0$,
\[
\p ( |N| \geq  x ) \leq \frac{\sqrt 2}{ x \sqrt \pi } \exp ( - x^2/2) \, , 
\]
we infer that 
\[
\p \Big (  \max_{1 \leq j \leq q_k} \big |  L^{1/2}_k ({\bm \eta}_{k,j+1 }) {\mathcal G}^{{\bm \eta}}_{k,1+j}  \big |  >  C k^{ 1 + 1/ \gamma}  \ \Big ) \ll  k^{-1} q_k  \exp ( - \kappa_4 k^2 ) \, , 
\]
where $\kappa_4$ is a positive constant depending on $\kappa_3$ and $C$.  This proves \eqref{(3.40)}. To end the proof of \eqref{neglivarphin}, it remains to prove that 
$ \sum_{k = K_0}^{b_n-1} L^{1/2}_k ({\bm \eta}_{k,q_k+1})  {\mathcal G}^{ {\bm \eta}}_{k,1+q_k} =  O( (\log n )^{\alpha} ) $ a.s. as $n \rightarrow \infty$. By Kronecker lemma, this holds if 
\[
 \sum_{k \geq  K_0}  k^{- \alpha} \E \big ( L^{1/2}_k ({\bm \eta}_{k,q_k+1})  |  {\mathcal G}^{ {\bm \eta}}_{k,1+q_k}  | \big )  < \infty \, .
\]
But
\[
\E \big ( L^{1/2}_k ({\bm \eta}_{k,q_k+1})  |  {\mathcal G}^{ {\bm \eta}}_{k,1+q_k}  | \big )  
= \int  \E \big ( L^{1/2}_k ({\bm \eta}_{k,q_k+1})  |  {\mathcal G}^{ {\bm \eta}}_{k,1+q_k}  | |  {\bm \eta} = {\bf u} \big ) dP_{ {\bm \eta} } ( {\bf u} ) \, .
\]
Hence, using the independence between $ {\mathcal G}^{{\bm \eta }}_{k,1+q_k} $ and ${\bm \eta}$ and the fact that $\E | {\mathcal G}^{{\bf u}}_{k,1+q_k} | \leq 1$, it follows that 
\[
\E \big ( L^{1/2}_k ({\bm \eta}_{k,q_k+1})  |  {\mathcal G}^{ {\bm \eta}}_{k,1+q_k}  | \big )  \leq  \E \big ( L^{1/2}_k (({\bm \eta}_{k,q_k+1})   \big )  \, .
\]Using the fact that the $\varepsilon_i$'s are iid, we infer that 
\[
\E \big ( L^{1/2}_k ({\bm \eta}_{k,q_k+1})   \big )  \leq 2 \Vert F^{(2)}_{k,q_k+1} \Vert_2 \, .
\]
But, by stationarity and the estimate (4.10) in \cite{CDKM18}, 
\[
\Vert F^{(2)}_{k,q_k+1} \Vert_2 \ll  \big \Vert \sum_{i=1}^{m_k} X_i  \big \Vert_2 + m_k  \big ( 2 \vert \psi  \vert_{\infty}\p ( T  \geq m_k) \big )^{1/2}+   \sum_{i=1}^{m_k} \Vert X_i  -X_{k,i} \Vert_2 \, .
\]
Hence, by Lemma \ref{lem:nngh} and Proposition~\ref{prop:boundofthedelta1first},
\[
\E \big ( L^{1/2}_k ({\bm \eta}_{k,q_k+1})   \big )   \ll   \big \Vert \sum_{i=1}^{m_k} X_i  \big \Vert_2 + k^{1/\gamma}  \exp(-\delta (2^{-1} \kappa)^{\gamma} k) \, .
\] 
Since by Lemma \ref{lmacovtildeXbis}, $\sum_{i \geq 0} | \Cov (X_0, X_i )| < \infty$, it follows that 
\[
\E \big ( L^{1/2}_k ({\bm \eta}_{k,q_k+1})   \big ) \ll \sqrt{m_k} \, .
\]
Hence
\[
 \sum_{k \geq  K_0}  k^{- \alpha} \E \big ( L^{1/2}_k ({\bm \eta}_{k,q_k+1})  |  {\mathcal G}^{ {\bm \eta}}_{k,1+q_k}  | \big ) \ll  \sum_{k \geq  K_0}  k^{- \alpha} k^{1/{(2 \gamma)}} < \infty \, ,
\]
since $\alpha - 1/{(2 \gamma)} = 1 + 1/{(2 \gamma)}  >1$. This ends the proof of \eqref{neglivarphin}. 

Now, the same arguments as to prove \eqref{neglivarphin}  show that 
\begin{equation} \label{negliGamman}
\max_{ i \leq n  }   |  \Gamma_i  ({\bm \eta}) | =O( (\log n )^{\alpha} )  \,  \text{ and then } \, \max_{ i \leq n  }   \big |  \sum_{k = K_0}^{h_i -1}  R^{{\bm \eta}}_{k,0}+  R^{{\bm \eta}}_{h_i,0} \big | =O( (\log n )^{\alpha} )    \, \text{ a.s.} 
\end{equation}
where we recall that ${\Gamma}_n ({\bf u}) $ has been defined in \eqref{defGammana}. So, overall, taking into account \eqref{conclusionp9},  \eqref{decompo1brownien},  \eqref{neglivarphin} and \eqref{negliGamman}, we get 
\begin{equation} \label{conclusionp9bis}
\max_{ i \leq n  }  \big |  \sum_{k = K_0}^{h_i-1} \sum_{j=1}^{q_k} R^{{\bm \eta}}_{k,j}+  \sum_{j=1}^{\tau_i} R^{{\bm \eta}}_{h_i,j} - {\overline \omega}_i( {\bm \eta}) \big | = O( (\log n )^{\alpha} ) \, \text{ a.s.}
\end{equation} 
But note now that 
\[
\Big (  \sum_{k = K_0}^{h_i-1} \sum_{j=1}^{q_k} R^{{\bm \eta}}_{k,j}+  \sum_{j=1}^{\tau_i} R^{{\bm \eta}}_{h_i,j}  + M_i ({\bm \eta}) \Big )_{i \geq N_0} =^{\mathcal D} ( S_i^{\diamond} )_{i \geq N_0} \, ,
\]
where 
\[
M_n ({\bf u}) =  \sum_{k = K_0}^{b_n-1} \sum_{j=1}^{q_k} \big \{  {\Lambda}_{k,1} ({{\bf u}}_{k,j})+ {\Lambda}_{k,3}({{\bf u}}_{k,j+1} )  \big \} +  \sum_{j=1}^{\tau_n} \big \{ 
 {\Lambda}_{b_n,1}({{\bf u}}_{k,j})+ {\Lambda}_{b_n,3}({{\bf u}}_ {b_n,j+1} ) \big \} \, .
\]
Recalling \eqref{samelawphiwi},  it remains to prove a strong invariance principle for ${\Phi}_n + M_n ({\bm \eta})$ (where $\Phi_n$ is defined in \eqref{defPhin}). With this aim, let 
\[
A_{k,j} = \big ( V_k^{\circ} ( {\bm \eta}_{k,j} ) \big )^{1/2} Z_{k,j}^{\star} + \Lambda_{k,1} ( {\bm \eta}_{k,j} ) +  \Lambda_{k,3}  ( {\bm \eta}_{k,j} ) \, ,
\]
where we recall that $ \Lambda_{k,1} ( {\bf u}_{k,j} )  = \E F^{(1)}_{k,j}  ( {\bf u}_{k,j} )$ and $ \Lambda_{k,3}  (  {\bf u}_{k,j} )  = \E F^{(3)}_{k,j}  ( {\bf u}_{k,j} )$. Note that the random variables $A_{k,j} $, $ j=1, \ldots q_k$, $k \geq K_0$ are independent. Denote by
\begin{equation} \label{defslashS}
S_n^{\sharp} := \sum_{k = K_0}^{b_n-1} \sum_{j=1}^{q_k}  A_{k,j}  + \sum_{j=1}^{\tau_n} A_{b_n,j} 
\end{equation}
and
\[
R_n^{\sharp} := {\Phi}_n + M_n ({\bm \eta}) - S_n^{\sharp} =   \sum_{k = K_0}^{b_n-1} \big \{  {\Lambda}_{k,3}({{\bm \eta}}_{k,q_k+1} ) - {\Lambda}_{k,3}({{\bm \eta}}_{k,1} ) \big \} +   \big \{ 
 {\Lambda}_{b_n,3}({{\bm \eta}}_{b_n, \tau_n+1} ) -  {\Lambda}_{b_n,3}({{\bm \eta}}_{b_n,1} ) \big \} \, .
\]
For any $j \geq 0$, note that
\[
 | {\Lambda}_{k,3}({{\bm \eta }}_{k,j+1} ) |  \leq 2 m_k |\psi|_{\infty} \ll k^{ 1/\gamma}  \text{ a.s. }
 \]
Therefore
\begin{equation} \label{negliRnsharp}
 \max_{1 \leq i \leq n}  |R_i^{\sharp}  | =O((\log n)^{ 1 + 1/\gamma}   ) \, \text{ a.s. }
\end{equation} 
Hence to prove the strong invariance principle for ${\Phi}_n + M_n ({\bm \eta})$ (and then for $  S_n^{\diamond}  $) with rate $O ((\log n)^{\alpha})$, it suffices to prove a strong invariance principle for $S_n^{\sharp}$ with the same rate. With this aim, recall that the random variables $A_{k,j} $, $ j=1, \ldots q_k$, $k \geq K_0$ are independent. We shall then use again Theorem 1 in Sakhanenko \cite{Sa84}. Note first that there exists a positive constant $\kappa_5$ depending only on $\delta$, $\gamma$ and 
$ |\psi|_{\infty}$  such that 
\[
 \big ( V_k^{\circ} ( {\bm \eta}_{k,j} ) \big )^{1/2} \leq \kappa_5 k^{1/\gamma} \ a.s. \, \text{ and } \,  \Lambda_{k,0} ( {\bm \eta}_{k,j} ) +  \Lambda_{k,2}  ( {\bm \eta}_{k,j} )  \leq \kappa_5 k^{1/\gamma} \ a.s.
\]
Hence, if we define for any integers $k \geq K_0$ and $j  \geq 1$, 
\[
\xi_{k,j} = k^{-1/ \gamma} A_{k,j} 
\]
we get, for any $t >0$, 
\begin{align*}
\E \big ( |\xi_{k,j} |^3 & \mathrm{e}^{t |\xi_{k,j} |}\big )
\\  & \leq 4 k^{-3/ \gamma}  \E \Big ( \big ( V_k^{\circ} ( {\bm \eta}_{k,j} ) \big )^{3/2}  |N|^3 \mathrm{e}^{t |\xi_{k,j} |} \Big ) +  4 k^{-3/ \gamma}  \E \Big ( \big |  \Lambda_{k,0} ( {\bm \eta}_{k,j} ) +  \Lambda_{k,2}  ( {\bm \eta}_{k,j} )\big |^{3} \mathrm{e}^{t |\xi_{k,j} |}\Big )  \\
 & \leq 4 \kappa_5 k^{-2/\gamma} \mathrm{e}^{t \kappa_5  }  \E \Big ( \big ( V_k^{\circ} ( {\bm \eta}_{k,j} ) \big )  |N|^3 \mathrm{e}^{t \kappa_5 | N  |}  \Big ) +  4 \kappa_5 k^{-2/\gamma} \mathrm{e}^{t \kappa_5 }  \E \Big ( \big |  \Lambda_{k,0} ( {\bm \eta}_{k,j} ) +  \Lambda_{k,2}  ( {\bm \eta}_{k,j} )\big |^{2} \mathrm{e}^{t \kappa_5  | N  |}\Big )  \, ,
\end{align*}
where $N$ is a standard Gaussian real-valued r.v. independent of $(\varepsilon_i)_{i \in {\mathbb Z}}$.  Therefore, it follows that for any $t >0$ and any integers $k \geq K_0$ and $j  \geq 1$, 
\begin{multline*}
\E \big ( |\xi_{k,j} |^3 \mathrm{e}^{t |\xi_{k,j} |}\big ) \\
\leq  4 \kappa_5 k^{-2/\gamma} \mathrm{e}^{t \kappa_5  } \E  \big ( V_k^{\circ} ( {\bm \eta}_{k,j} ) \big )  \E \Big (  |N|^3 \mathrm{e}^{t \kappa_5 | N  |}  \Big ) +   4 \kappa_5 k^{-2/\gamma} \mathrm{e}^{t \kappa_5  } \E \Big ( \big |  \Lambda_{k,0} ( {\bm \eta}_{k,j} ) +  \Lambda_{k,2}  ( {\bm \eta}_{k,j} )\big |^{2}  \Big ) \E ( \mathrm{e}^{t \kappa_5  | N  |}\Big )  \, .
\end{multline*}
Since $ Z^{\star}_{k,j} $, $k,j \in {\mathbb Z}$ are centered with variance $1$ and  independent of $(\varepsilon_i)_{i \in {\mathbb Z}}$, it follows that 
\[
\E (A^2_{k,j})= \E \big ( V_k^{\circ} ( {\bm \eta}_{k,j} ) \big )  +  \E \Big ( \big |  \Lambda_{k,1} ( {\bm \eta}_{k,j} ) +  \Lambda_{k,3}  ( {\bm \eta}_{k,j} )\big |^{2}  \Big ) \, .
\]
Therefore,  for any $t >0$ and any integers $k,j$,  
\[
t \E \big ( |\xi_{k,j} |^3 \mathrm{e}^{t |\xi_{k,j} |}\big ) 
\leq 4 t  \kappa_5 \mathrm{e}^{t \kappa_5  } \E (\xi^2_{k,j})   \E \Big (  ( |N|^3  \vee 1) \mathrm{e}^{t \kappa_5 | N  |}  \Big )  \, .
\]
Hence, taking into account \eqref{trivialgaussian} and the fact that $\E \big ( \mathrm{e}^{t \kappa_5 | N  |}  \big ) \leq 2 \mathrm{e}^{t^2 \kappa_5^2   /2}$, it follows that there exists a positive constant  $\kappa_6$  depending only on $\kappa_5$ such that for $t = \kappa_6 $,
\[
t \E \big ( |\xi_{k,j} |^3 \mathrm{e}^{t |\xi_{k,j} |}\big ) 
\leq  \E (\xi^2_{k,j})   \, .
\]
Using Theorem 1 in Sakhanenko \cite{Sa84}  and \eqref{AbelShao}, we then infer that there exists a Brownian motion ${\mathbb B}$ such that 
\begin{equation*} 
 \max_{1 \leq i \leq n} \Big | S_i^{\sharp} -{\mathbb B}(\sigma_i^2) \Big | =O( (\log n )^{\alpha}  ) \, \text{ a.s. }
\end{equation*}
where $\sigma_n^2$ is defined by  \eqref{defsigmansquare}. This ends the proof of \eqref{(3.48)} and then of Step 4. 

\subsection{Step 5. (\textit{Identifying the variance in the Brownian motion})}

The aim of this step is to show that it is possible to replace in \eqref{(3.48)} the variance function $\sigma_i^2$ by $i c^2$. A careful analysis of Step 3.4 in \cite{BLW14} reveals that this holds provided that setting $\nu_k := \Vert A_{k,1} \Vert_2^2 /(6m_k)$, 
\begin{equation} \label{cond1v_k}
( \log n)  \max_{ \ell \leq b_n} (m_\ell \nu_\ell)^{1/2} = O( (\log n )^{\alpha}  ) \, ,  
\end{equation}
and
\begin{equation} \label{cond2v_k}
3^\ell (  \nu_\ell^{1/2}  - c)^2 = O( \ell^{2 \alpha} (\log \ell)^{-1}) \, .  
\end{equation}
Note also that since $c^2$ is assumed to be positive, to prove \eqref{cond2v_k}, it suffices to prove that 
\begin{equation} \label{cond2v_kbis}
3^\ell (  \nu_\ell - c^2)^2 = O( \ell^{2 \alpha} (\log \ell)^{-1})  \, , \, \mbox{ as $\ell \rightarrow \infty$}\, . 
\end{equation}
Before proving the above convergences, we first notice that by Lemma \ref{lmacovtildeXbis}, $\sum_{i \geq 0} | \Cov (X_0, X_i )| < \infty$. Hence
\[
c^2 =  \lim_{n \rightarrow \infty } \frac{1}{n} \int |S_n (\varphi )|^2 \, d\mu =  \lim_{n \rightarrow \infty } \frac{1}{n}  \Vert S_n \Vert_2^2 = {\rm Var} (X_0) +2\sum_{i \geq 0}  \Cov (X_0, X_i ) \, .
\]
Moreover,  proceeding as to get the relation \cite[(66)]{CDM17}, we have
\[
\nu_{\ell} =  {\tilde c}_{\ell, 0}  + 2 \sum_{k =1}^{2m_\ell }  {\tilde c}_{\ell,k}  \, , 
\]
where, for any $i \geq 0$, 
\[
{\tilde c}_{\ell,i}= \Cov ({\tilde X}_{\ell, m_\ell+1} , {\tilde X}_{\ell,i+m_\ell+1} ) \, .
\]
Note that, by stationarity, for all $i \geq 0$, 
\begin{multline*}
\big |  {\tilde c}_{\ell,i} - \Cov  ( X_0, X_i ) \big |  = \big | \Cov ({\tilde X}_{\ell, m_\ell+1}  - X_{m_\ell+1}, {\tilde X}_{\ell,i+m_\ell+1} )   +  \Cov  ( X_{m_\ell+1} , 
{\tilde X}_{\ell,i+m_\ell+1}  - X_{i+m_\ell+1} ) \big | \\
\leq 2 \vert \psi \vert_{\infty} \big ( \Vert {\tilde X}_{\ell, m_\ell+1}  - X_{m_\ell+1} \Vert_1+ \Vert  {\tilde X}_{\ell, i+m_\ell+1}  - X_{i+m_\ell+1} \Vert_1 \big ) \, .
\end{multline*}
But by  the estimate (4.10) in \cite{CDKM18}, Proposition~\ref{prop:boundofthedelta1first} and Lemma \ref{lem:nngh}, for any $k > m_{\ell}$, 
\[
 \Vert {\tilde X}_{\ell, k } - X_k \Vert_1 \ll  \bP  ( T \geq  m_{\ell} )  + \exp (-\delta m_{\ell}^{\gamma} ) \ll \exp (- \delta  (2^{-1} \kappa)^{\gamma} \ell  ) \, .
\]
Moreover, according to Lemma \ref{lmacovtildeXbis}, 
\[
\sum_{k > 2m_\ell} |\Cov (X_0, X_k ) | \ll  \exp ( -\delta m_{\ell}^{\gamma} /2 )  \ll  \exp ( -\delta (2^{-1} \kappa)^{\gamma} \ell/2  )  \, .
\]
Therefore, setting ${\tilde c} =  \delta (2^{-1} \kappa)^{\gamma}  /2 $, 
\begin{equation} \label{cond2v_kbisBound}
| \nu_{\ell}  -  c^2 | \ll  \exp (-{\tilde c} \ell  )  \, .
\end{equation}
This shows that $\nu_\ell \rightarrow c^2 $, as $\ell \rightarrow \infty$. Hence \eqref{cond1v_k}  is satisfied since $( \log n)  \max_{ k \leq b_n} (m_k )^{1/2} \ll (\log n )^{1+ 1/(2 \gamma) } = O( (\log n )^{\alpha}  ) $. Now, \eqref{cond2v_kbisBound} proves \eqref{cond2v_kbis} since $ 2 {\tilde c} \geq \log 3 $. The proof of step 5 is complete. This ends the proof of the theorem when  $c^2>0$.  $\square$

%
%

\appendix

\section{Example of nonuniformly expanding system with stretched exponential return times}
\label{sec:example}

Suppose that $f \colon X \to X$ is a nonuniformly expanding dynamical system with
base $Y \subset X$, reference measure $m$ on $Y$ and return time $\tau \colon Y \to \N$.
In all standard examples, the return time tails
$m(\tau > n)$ decay exponentially, except for Alves-Viana maps with
$m(\tau > n) = O(\mathrm{e}^{-c \sqrt{n}})$ and intermittent maps~\cite{LSV99} with polynomial decay.

Here we present a family
of interval maps with an optimal bound on the return times
$m(\tau > n) \sim \mathrm{e}^{-\kappa n^\gamma}$,
where $\gamma \in ]0,1]$ is a parameter and $\kappa = \kappa(\gamma) >0$.
For this class of maps, Theorem~\ref{thm:NUE} gives all possible subexponential rates
in the ASIP.

Fix $\gamma \in ]0,1]$ and consider the following modification of the intermittent maps
from~\cite{LSV99}.  Let \(f \colon [0,1] \to [0,1]\),
\begin{equation} \label{eq:LSV}
 f(x) = \begin{cases}
    x\bigl(1 + \frac{c}{|\log x|^\beta}\bigr), & x \leq 1/2 \\
    2 x - 1, & x > 1/2
  \end{cases}
\end{equation}
with $\beta = \gamma^{-1} - 1$ and $c=(\log 2)^\beta$ so that $f(1/2) = 1$.

Let $Y = ]1/2,1]$ be a base, $\tau \colon Y \to \N$, $\tau(x) = \inf \{k \geq 1 : f^k(x) \in Y\}$
be the first return time and $F \colon Y \to Y$, $F(x) = f^{\tau(x)}(x)$ be the induced map.
Let $\alpha$ denote the partition of $Y$ into the intervals where $\tau$ is constant.
Let $m$ denote the Lebesgue measure.

In the rest of this section we prove:

\begin{thm}
  \label{thm:ex}
  $f$ is a nonuniformly expanding map with base $Y$, return time $\tau$ and
  reference measure $m$. That is, there exists $C > 0$ such that
  for every $a \in \alpha$ and all $x,y \in a$,
  \begin{enumerate}[label=(\alph*)]
    \item\label{thm:ex:bi} $F \colon a \to Y$ is a nonsingular bijection;
    \item\label{thm:ex:exp} $F$ is expanding: $|F(y) - F(x)| \geq 2 |y-x|$;
    \item\label{thm:ex:dist} $F$ has bounded distortion: $|\log F'(y) - \log F'(x)| \leq C |F(y) - F(x)|$.
  \end{enumerate}
  Further, there exist $\eta_1, \eta_2 > 0$ such that for all $n \geq 1$,
  \begin{enumerate}[label=(\alph*),resume]
    \item\label{thm:ex:tails}
    $
      {\rm e}^{-\eta_2n^\gamma }
      \leq m(\tau \geq n) 
      \leq  {\rm e}^{-\eta_1n^\gamma }
    $.
  \end{enumerate}
\end{thm}

Proof of Theorem~\ref{thm:ex} takes the rest of this section.
Items~\ref{thm:ex:bi} and~\ref{thm:ex:exp} are straightforward. For~\ref{thm:ex:dist}
and~\ref{thm:ex:tails}, we use the following technical lemma.

Let $g=f_{]0,1/2]}^{-1}$ be the inverse left branch of $f$. For $n\ge 0$, let $z_n= g^n$
and $u_n=-\log z_n$. 

\begin{lemma}
  \label{lem:aaga}
  There exist $\delta_1>\delta_2>0$ such that for every $n\ge 1$, 
  \begin{equation} \label{control-un}
    \delta_2n^\gamma \le u_n(x) \le \delta_1 n^\gamma
    \qquad
    \forall x\in (1/2,1]\, .
  \end{equation}
  Further, there exists $C>0$ such that for every $n\ge 1$, 
  \begin{equation}
    \label{dist}
    |\log z_n'(x) -\log z_n'(y)|\le C |x-y| \qquad \forall x,y\in (1/2,1]\, .
  \end{equation}
\end{lemma}

\begin{proof}
  We have 
  \begin{equation}
    \label{recurrence}
    u_n=u_{n+1} - \log\bigl(1+c/u_{n+1}^\beta\bigr)
    \, .
  \end{equation}
  Observe that $(x_n)$ is decreasing to 0 and $(u_n)$ is increasing to $\infty$. Hence,
  \begin{align*}
    n+1 = \sum_{k=0}^{n} \frac{u_{k+1}-u_k}{\log(1+c u_{k+1}^{-\beta})}  
    & \ge \sum_{k=0}^{n}
      \int_{u_k}^{u_{k+1}}\frac{dx}{\log (1+c x^{-\beta})}
    \\
    & =\int_{u_0}^{u_{n+1}}\frac{dx}{\log (1+c x^{-\beta})}
    \ge K u_{n+1}^{\beta+1}
    \, ,
  \end{align*}
  for some $K>0$ not depending on $n$.
  Recall that $\gamma = (1 + \beta)^{-1}$.
  Thus $u_n \le  K^{-1} n^\gamma$.  

  By~\eqref{recurrence}, $u_{n+1}/u_n\to 1$, hence for some $c', K' > 0$,
  \begin{align*}
    n+1 \leq \sum_{k=0}^{n} \frac{u_{k+1}-u_k}{\log(1+c' u_{k}^{-\beta})}  
    & \leq \sum_{k=0}^{n}
      \int_{u_k}^{u_{k+1}}\frac{dx}{\log (1+c' x^{-\beta})}
    \\
    & =\int_{u_0}^{u_{n+1}}\frac{dx}{\log (1+c' x^{-\beta})}
    \leq K' u_{n+1}^{\beta+1}
    \, .
  \end{align*}
  Thus $u_n \ge  K'^{-1} n^\gamma$.
  This completes the proof of~\eqref{control-un}. 

  It remains to prove~\eqref{dist}. It suffices to show that
  \[
    \sup_{n\ge 1} \sup_{x\in (1/2,1]} 
    \Bigl|\frac{z_n''(x)}{z_n'(x)}\Bigr|
    <\infty
    \, .
  \]
  Let $n\ge 1$. We have
  \begin{align*}
    z_n & = z_{n+1}\biggl(1+\frac{c}{|\log z_{n+1}|^\beta}\biggr)
    \, \\
    z_n' & = z_{n+1}'\biggl(1+\frac{c}{|\log z_{n+1}|^\beta}
      + \frac{c\beta}{|\log z_{n+1}|^{\beta+1}}\biggr)
    \, \\
    z''_n & = z''_{n+1}
      \biggl(
        1+\frac{c}{|\log z_{n+1}|^\beta}
        + \frac{c\beta}{|\log z_{n+1}|^{\beta+1}}
      \biggr)
      + \frac{(z'_{n+1})^2}{z_{n+1}}
      \biggl(
        \frac{c\beta}{|\log z_{n+1}|^{\beta+1}}
        + \frac{c\beta(\beta+1)}{|\log z_{n+1}|^{\beta+2}}
      \biggr)
    \, .
  \end{align*}
  By the above computations and~\eqref{control-un}, we have 
  \[
    \Bigl| \frac{z'_{n}}{z_{n}} \Bigr|
    = \Bigl| \frac{z'_{n+1}}{z_{n+1}} \Bigr|
      \biggl(1+\frac{c\beta}{
        |\log z_{n+1}|^{\beta+1}(1+c|\log z_{n+1}|^{-\beta})
      }\biggr)
    \geq \Bigl| \frac{z'_{n+1}}{z_{n+1}}\Bigr|
        \Bigl(1+\frac\varepsilon {n+1}\Bigr)
  \]
  for some $\varepsilon>0$. 
  Hence
  \begin{equation}\label{second-est}
    \Bigl| \frac{z'_n}{z_n}\Bigr| 
    \leq \Bigl|\frac{z'_0}{z_0}\Bigr| \prod_{k=1}^n \Bigl(1 + \frac{\eps}{k}\Bigr)^{-1}
    \le \frac{L}{n^\eta}
  \end{equation}
  for some $L>0$ and $\eta>0$.
  
  Next, there is $K>0$ so that
  \begin{equation}\label{first-est}
    \Bigl|\frac{z_{n+1}''}{z_{n+1}'}\Bigr|
    \le \Bigl|\frac{z_{n}''}{z_{n}'}\Bigr|
     + \frac{K}{n+1} \Bigl| \frac{z'_{n+1}}{z_{n+1}}\Bigr|\, .
  \end{equation}
  Combining \eqref{second-est} and \eqref{first-est}, we obtain~\eqref{dist}.
\end{proof}

Observe that if $x \in Y$ with $\tau(x) = n$, then $F'(x) = 2 / z_{n-1}'(F(x))$.
With~\eqref{dist}, this proves Theorem~\ref{thm:ex}~\ref{thm:ex:dist}.
Further, $m(\tau \geq n) = z_n(1)/2 = \mathrm{e}^{-u_n(1)}$.
Given~\eqref{control-un}, this implies Theorem~\ref{thm:ex}~\ref{thm:ex:tails}.

The proof of Theorem~\ref{thm:ex} is complete. $\square$

\section{Proof of Proposition \ref{prop:boundofthedelta1first}}
\label{sec:prop:boundofthedelta1first}

Recall that to simplify the exposition and notation, we assume that $p = \gcd \{h(w) \colon w \in \cA\}=1$, so that the underlying Markov chain $(g_k)_{k \geq 0}$ is aperiodic. 

 Note first that, as quoted in \cite{CDKM18}, $\psi$ has the following property: For \(a,b \in \Omega\), with $a=  (g_0, \ldots, g_N, g_{N+1}, \ldots)$, $
  b  = (g_0, \ldots, g_N, g'_{N+1}, \ldots) $ 
with \(g_{N+1} \neq g'_{N+1}\),
\[
|\psi (a) - \psi (b)| \leq C \theta^{\sum_{k=0}^N {\bf 1}_{\{g_k \in S_0\}}} \, , 
\]
where $C$ is a constant depending on $\lambda$ (the constant appearing in \eqref{distanceseparation}), the diameter of $X$ and on the H\"older norm  of  $\varphi$,  and $\theta \in ]0,1[$ depends on $\lambda$ and on the H\"older exponent of $\varphi$.  
Therefore  \(
    \delta_\ell \leq C \theta^{s_\ell}
  \),
  where $s_\ell = \#\{k \leq \ell \colon g_k  \in S_0\}$. Let now $S_c = \{ ( w, h(w) -1) , w \in  \cA \}$.  Note that $S_c$ is an atom for the Markov chain $(g_n, n \geq 0)$ in the sense that if the chain enters in $S_c$ then it regenerates. Moreover, 
  \[
   \begin{aligned}
   s_\ell = \sum_{i=0}^{\ell}  {\bf 1}_{\{g_i \in S_0 \}}   & = \begin{cases}
     \sum_{i=0}^{\ell -1}  {\bf 1}_{\{g_i \in S_c \}}   & \text{ if  }g_0 \notin S_0\\
   1+  \sum_{i=0}^{\ell -1}  {\bf 1}_{\{g_i \in S_c \}}  &  \text{ if  }g_0 \in S_0 \, .
  \end{cases}
  \end{aligned}
   \]
 Hence,  note that
 \[
 C^{-1}\bE(\delta_\ell)
       \leq
      \bE ( \theta^{s_\ell}) \leq  \bE ( \theta^{\sum_{i=0}^{\ell -1}  {\bf 1}_{\{g_i \in S_c \}} }) = \bE ( \theta^{\sum_{i=1}^{\ell }  {\bf 1}_{\{g_i \in S_c \}} }) \, . \]
Let $R_0 = \inf \{ n > 0 \, : \, g_n \in S_c \}$ be the first renewal time and for $i \geq 1$,
  \[
  R_i = \inf \{ n > R_{i-1} \, : \, g_n \in S_c \}  \text{ and } \tau_i = R_i - R_{i-1} \, .
  \]
  Note that $(\tau_i)_{i \geq 1}$ forms a sequence of iid random variables and their common law is the law of $R_0$ when the chain starts from $S_c$. Let $\kappa = 1/(4 \E ( \tau_1))$. We have
    \begin{equation}
    \label{dec1step1}
      \bE ( \theta^{\sum_{i=1}^{\ell }  {\bf 1}_{\{g_i \in S_c \}} }) 
      \leq \theta^{  \kappa \ell }
            + \bP \Bigl( \sum_{i=1}^{\ell }  {\bf 1}_{\{g_i \in S_c \}} <  \kappa \ell \Bigr)
      \, .
  \end{equation}
  Next,
  \[
    \bP \Bigl( \sum_{i=1}^{\ell }  {\bf 1}_{\{g_i \in S_c \}} <  \kappa \ell \Bigr) \leq  \bP \Bigl( R_{ [\kappa \ell ] +1} > \ell \Bigr) \, .
  \]
Note that $R_{ [\kappa \ell ] +1} =  R_0 + \sum_{i=1}^{ [\kappa \ell ] +1} \tau_i $.  Hence
   \[
   \bP \Bigl( R_{ [\kappa \ell ] +1} > \ell \Bigr) \leq   \bP \Bigl( R_{0} > \ell /2 \Bigr)  +  \bP \Bigl( \sum_{i=1}^{ [\kappa \ell ] +1} \tau_i> \ell /2 \Bigr)   \, .
  \]
But 
\[
  \bP \Bigl( \sum_{i=1}^{ [\kappa \ell ] +1} \tau_i> \ell /2 \Bigr)  \leq    \bP \Bigl( \sum_{i=1}^{ [\kappa \ell ] +1} ( \tau_i - \E ( \tau_i)) > \ell /4 - \E ( \tau_1)  \Bigr) 
    \leq    \bP \Bigl( \sum_{i=1}^{ [\kappa \ell ] +1} ( \tau_i - \E ( \tau_i) )  > \ell /8   \Bigr)  \, , 
\]
if $\ell \geq 8 \E ( \tau_1 )$.  Now, using  that $\bP ( \tau_1 \geq k) = \bP_{\mathcal A} ( h \geq k) $ and that if $g_0 = ( w , \ell)$, then $R_0 = h(w) - \ell-1$, and recalling that since we assume a subexponential moment of order $\gamma$ for the return time, $h$  has also a subexponential moment of the same order, we  infer that there exists a positive constant $c$ such that 
\[
\bE ( {\rm e}^{t \tau_1^{\gamma}} )  < \infty \text{ and } \bE_{\nu} ( {\rm e}^{t R_0^{\gamma}}) < \infty \text{ for any $|t| \leq c $} \, .
\]
According to  Bernstein's $\psi_1$ inequality (see for instance \cite[Lemma 2.2.11]{VW96} and the subsequent
remark) when $\gamma =1$ or to the proof of Corollary 5.1 in Borovkov  \cite{Bo00} when $\gamma <1$ (see also inequality (1.4) in \cite{MPR}), there exists a positive constant $c_1$ such that 
\[
 \bP \Bigl( \sum_{i=1}^{ [\kappa \ell ] +1} ( \tau_i - \E ( \tau_i) )  > \ell /8   \Bigr)  \leq \exp ( -c_1 \ell^{\gamma} ) \, .
\]
So, overall, there exists a positive constant $c_2$ such that 
 \begin{equation}
    \label{dec1step2}
    \bP \Bigl( R_{ [\kappa \ell ] +1} > \ell \Bigr) \leq  \exp ( -c_2 \ell^{\gamma} ) 
      \, .
  \end{equation}
The proposition follows by taking into account \eqref{dec1step1} and \eqref{dec1step2}.  $\square$

\subsection*{Acknowledgements}
C.C.\ acknowledges the hospitality of Warwick University and Universit\'e Paris-Est (Marne-la-Vall\'ee).
A.K.\ was supported in part by a European Advanced Grant {\em StochExtHomog} (ERC AdG 320977)
at the University of Warwick and an Engineering and Physical Sciences Research Council grant
EP/P034489/1 at the University of Exeter, and is grateful to Mark Holland and Ian Melbourne
for support during this work.

\end{document}